\documentclass[12pt,  reqno]{amsart}

\usepackage{amsfonts}
\usepackage{amscd}
\usepackage{amsmath,amssymb,amsthm} 
\usepackage{mathrsfs}
\usepackage{enumerate}
\usepackage{float} 
\usepackage[pdftex]{graphicx}
\allowdisplaybreaks

\usepackage{hyperref}

\usepackage{color}

\theoremstyle{plain}
\newtheorem{thm}{Theorem}[section]
\newtheorem{lem}[thm]{Lemma}

\theoremstyle{definition}

\newtheorem*{thm*}{Theorem}

\theoremstyle{remark}
\newtheorem{rem}{Remark}[section]

\numberwithin{equation}{section}

\newcommand{\He}{\mathbb{H}}
\newcommand{\C}{\mathbb{C}}
\newcommand{\R}{\mathbb{R}}

\usepackage{color}

\title[ Ingham's theorem on the Heisenberg group]
{An analogue of Ingham's theorem \\
	on the Heisenberg group}

\author[Bagchi, Ganguly, Sarkar and Thangavelu]
{Sayan Bagchi, Pritam Ganguly, Jayanta Sarkar\\ and Sundaram Thangavelu}

\address[S. Bagchi, J. Sarkar]{Department of Mathematics and Statistics\\
	Indian Institute of Science Education and Research Kolkata\\
	Mohanpur-741246, Nadia, West Bengal, India.} 
\email{sayansamrat@gmail.com, jayantasarkarmath@gmail.com}

\address[P. Ganguly, S. Thangavelu]{Department of Mathematics, Indian Institute of Science,  Bangalore-560 012, India.}
\email{pritam1995.pg@gmail.com, veluma@iisc.ac.in}


\begin{document}
	\begin{abstract}  
		We prove an exact analogue of Ingham's uncertainty principle for the group Fourier transform on the Heisenberg group. This is accomplished by explicitly constructing  compactly supported functions  on the Heisenberg group whose operator valued Fourier transforms have suitable Ingham type decay and proving an analogue  of Chernoff's theorem for the family of special Hermite operators. 
	\end{abstract}
	\maketitle
	
	\section{Introduction} Roughly speaking, the uncertainty principle for the Fourier transform on $ \R^n $ says that a function $ f $ and its Fourier transform $\widehat{f} $ cannot both have rapid decay. Several manifestations of this principle are known: Heisenberg-Pauli-Weyl inequality, Paley-Wiener theorem and Hardy's uncertainty principle are some of the most well known. But there are lesser known results such as theorems of Ingham and  Levinson. The best decay a non trivial function can have is vanishing identically outside a compact set and for such functions it is well known that their Fourier transforms extend to $ \C^n $ as entire functions and hence cannot vanish on any open set. For  any such function of compact support, its Fourier transform cannot have any exponential decay for a similar reason: if $ |\widehat{f}(\xi)| \leq C e^{-a|\xi|} $ for some $ a > 0 $, then it follows that $ f $ extends to a tube domain in $ \C^n $ as a holomorphic function and hence it cannot have compact support. So it is natural to ask the question: what is the best possible decay, on the Fourier transform side, that is allowed of a function of compact support? An interesting answer to this question is provided by the following theorem of Ingham \cite{I}.
	
	\begin{thm}[Ingham]\label{ingh}
		Let $ \Theta(y) $ be a nonnegative even function on $\R$ such that $ \Theta(y) $ decreases to zero when $ y \rightarrow \infty.$ There exists a nonzero continuous function $ f $ on $ \R,$ equal to zero outside an interval $ (-a,a) $ whose Fourier transform $ \widehat{f} $ satisfies the estimate $ |\widehat{f}(y)|\leq C e^{-|y|\Theta(y)} $ if and only if  $ \int_1^\infty \Theta(t) t^{-1} dt <\infty.$
	\end{thm}
	
	This theorem of Ingham and its close relatives   Paley -Wiener (\cite{PW1, PW2})  and Levinson (\cite{Levinson}) theorems have received considerable attention in recent years. In \cite{BRS}, Bhowmik et al proved analogues of the above theorem for $ \R^n,$ the $n$-dimensional torus $ \mathbb{T}^n $  and step two nilpotent Lie groups. See also the recent work of Bowmik-Pusti-Ray \cite{BPR} for a version of Ingham's theorem for the Fourier transform on Riemannian symmetric spaces of non-compact type. As we are interested in Ingham's theorem on the Heisenberg group, let us recall the result proved in \cite{BRS}. Let $ \He^n = \C^n \times \R $ be the Heisenberg group. For an integrable function $ f $ on $ \He^n $, let $ \widehat{f}(\lambda) $ be the operator valued Fourier transform of $ f $ indexed by non-zero reals $ \lambda.$  Measuring the decay of the Fourier transform in terms of the Hilbert-Schmidt operator norm $ \|\widehat{f}(\lambda)\|_{HS} $ Bhowmik et al. have proved the following result.
	
	\begin{thm}[Bhowmik-Ray-Sen]\label{BRS}
		Let $ \Theta(\lambda) $ be a nonnegative even function on $\R$ such that $ \Theta(\lambda) $ decreases to zero when $ \lambda \rightarrow \infty.$ There exists a nonzero, compactly supported  continuous function $ f $ on $ \He^n,$  whose Fourier transform satisfies the estimate $ \|\widehat{f}(\lambda)\|_{HS} \leq C |\lambda|^{n/2}e^{-|\lambda| \Theta(\lambda)} $ if  the integral  $ \int_1^\infty \Theta(t) t^{-1} dt <\infty.$ On the other hand, if the above estimate is valid for a function $ f $ and the integral $ \int_1^\infty \Theta(t) t^{-1} dt $ diverges, then the vanishing of $ f $ on any set of the form $ \{z\in \C^n: |z|< \delta \} \times \R $ forces $ f $ to be identically zero.
	\end{thm}
	
	As the Fourier transform on the Heisenberg group is operator valued, it is natural to measure the decay of $ \widehat{f}(\lambda) $ by comparing it with the Hermite semigroup $ e^{-aH(\lambda)} $ generated by $ H(\lambda)= -\Delta_{\mathbb{R}^n}+\lambda^2 |x|^2.$ In this connection, let us recall the following two versions of Hardy's uncertainty principle. Let $ p_a(z,t) $ stand for the heat kernel associated to the sublaplacian $ \mathcal{L} $ on the Heisenberg group whose Fourier transform turns out to be  the Hermite semigroup $ e^{-aH(\lambda)} .$ The version in which one measures the decay of $ \widehat{f}(\lambda) $ in terms of its Hilbert-Schmidt operator norm reads as follows. If
	\begin{equation}
		|f(z,t)| \leq C e^{-a(|z|^2+t^2)} ,\,\,  \|\widehat{f}(\lambda)\|_{HS} \leq C e^{-b \lambda^2 } 
	\end{equation}
	then $ f = 0 $ whenever $ ab > 1/4.$ This is essentially a theorem in the $ t$-variable and can be easily deduced from Hardy's  theorem on $ \R $, see Theorem 2.9.1 in \cite{TH3}. Compare this with the following version \cite[Theorem 2.9.2]{TH3}. If
	\begin{equation}
		|f(z,t)| \leq C p_a(z,t) ,\,\,   \widehat{f}(\lambda)^\ast \widehat{f}(\lambda) \leq C e^{-2bH(\lambda)} 
	\end{equation}
	then $ f = 0 $ whenever $ a < b.$ This latter version is the exact analogue of Hardy's theorem for the Heisenberg group, which we can view not merely as an uncertainty principle but also as a characterization of the heat kernel. Hardy's theorem in the context of semi-simple Lie groups and non-compact Riemannian symmetric spaces are also to be viewed in this perspective.

	We remark that the  Hermite semigroup has been used to measure the decay  of the Fourier transform in connection with the heat kernel transform \cite{KTX}, Pfannschmidt's theorem \cite{TH4} and the extension problem for  the sublaplacian \cite{RT} on the Heisenberg group. In connection with the study of Poisson integrals, it has been noted in \cite{ TH5} that when the Fourier transform of $ f $ satisfies an estimate of the form $\widehat{f}(\lambda)^\ast \widehat{f}(\lambda) \leq C e^{-a \sqrt{H(\lambda)}} ,$ then the function extends to a tube domain in the complexification of $ \He^n $ as a holomorphic function and hence the vanishing of $ f $ on an open set forces it to vanish identically. It is therefore natural to ask if the same conclusion can be arrived at by replacing the constant $ a $ in the above estimate by an operator $ \Theta(\sqrt{H(\lambda)}) $ for a function $ \Theta $ decreasing to zero at infinity. Our investigations have led us to the following exact analogue of Ingham's theorem for the Fourier transform on $ \He^n.$ 
	
	\begin{thm}\label{ingh-hei}
		Let $ \Theta(\lambda) $ be  a nonnegative  function on $ [0,\infty)$ which  decreases to zero as $ \lambda \rightarrow \infty .$ Then there exists a nonzero compactly supported continuous function $ f $ on $ \He^n$ whose  Fourier transform $ \widehat{f} $ satisfies  the estimate 
		\begin{equation}
			\label{fdecay} \widehat{f}(\lambda)^\ast \widehat{f}(\lambda) \leq C  e^{-2\Theta(\sqrt{H(\lambda)})\sqrt{H(\lambda)}},~\lambda\neq0, 
		\end{equation}
	if and only if $\Theta$ satisfies the condition $ \int_1^\infty \Theta(t) t^{-1} dt <\infty.$
	\end{thm}
	Under the assumption that $\int_1^\infty \Theta(t) t^{-1} dt =\infty$, the above theorem demonstrates that any compactly supported function whose Fourier transform satisfies \eqref{fdecay} vanishes identically. This can be viewed as an uncertainty principle in the sense mentioned in the first paragraph. Recently this aspect of Ingham's theorem has been proved in the context of higher dimensional Euclidean spaces and Riemannian symmetric spaces  with a much weaker hypothesis on the function. As observed in \cite{GT-heisenberg}, for the Heisenberg group case, the hypothesis can be weakened considerably if we slightly strengthen the condition \eqref{fdecay}. More precisely, the second and the last author proved the following theorem in this context. 
	\begin{thm}\label{ingh-hei-1}\cite{GT-heisenberg}
		Let $ \Theta(\lambda) $ be  a nonnegative  function on $ [0,\infty)$ such that it  decreases to zero as $ \lambda \rightarrow \infty $,  and satisfies the conditions $  \int_1^\infty \Theta(t) t^{-1} dt = \infty.$  Let $ f $ be  an integrable  function on $ \He^n $ whose Fourier transform  satisfies  the estimate \begin{equation}
			\label{fdecay-strong}
			\hat{f}(\lambda)^\ast \hat{f}(\lambda) \leq C \, e^{-2|\lambda|\,\Theta(|\lambda|)}e^{-2 \sqrt{H(\lambda)} \,\Theta(\sqrt{H(\lambda))}}.
		\end{equation} Then $ f $ cannot vanish  on any  nonempty open set unless it is identically zero.
\end{thm}  Comparing the decay condition \eqref{fdecay} and \eqref{fdecay-strong}, it is not difficult to see that the Theorem \ref{ingh-hei} is a significant improvement of the Theorem \ref{ingh-hei-1} in terms of the Ingham type decay condition. However, we believe that the necessary part of the Theorem \ref{ingh-hei} is true under the weaker hypothesis on the function as in the Theorem \ref{ingh-hei-1}. In what follows, we shed more light on the difficulties in this regard.     

The sufficiency part of Theorem \ref{ingh-hei} is proved in Section 4.1 by explicitly constructing compactly supported functions whose Fourier transforms satisfy the stated decay condition. Though at present we are not able to prove the necessary part of the theorem under the assumption that $ f $ vanishes on an open set, a slightly different version can be proved. Recall that the Fourier transform $ \widehat{f} $ is defined by integrating $ f $ against the Schr\"odinger representations $ \pi_\lambda$:
$$ \widehat{f}(\lambda) = \int_{\He^n} f(z,t) \pi_\lambda(z,t) dz\, dt .$$
Since $ \pi_\lambda(z,t) = e^{i\lambda t} \, \pi_\lambda(z,0) $, it follows that $ \widehat{f}(\lambda) = \pi_\lambda(f^\lambda), $ where $ f^\lambda(z) $ is the inverse Fourier transform of $ f(z,t) $ in the central variable and 
$$ W_\lambda(f^\lambda) = \int_{\C^n} f^\lambda(z) \pi_\lambda(z,0) dz $$ 
is the Weyl  transform of $ f^\lambda$. With these notations we prove the following improvement on the necessary part of Theorem \ref{ingh-hei}.

\begin{thm}\label{ingh-hei-variant}
		Let $ \Theta(\lambda) $ be  a nonnegative  function on $ [0,\infty)$ such that it  decreases to zero when $ \lambda \rightarrow \infty $, and satisfies the condition $  \int_1^\infty \Theta(t) t^{-1} dt = \infty.$ Let  $ f $ be an integrable function on $ \He^n$ whose  Fourier transform $ \widehat{f} $ satisfies  the estimate 
		\begin{equation}
			\label{fdecay-1} \widehat{f}(\lambda)^\ast \widehat{f}(\lambda) \leq C  e^{-2\Theta(\sqrt{H(\lambda)})\sqrt{H(\lambda)}},~\lambda\neq 0. 
		\end{equation}
If for every $ \lambda \neq 0,$ there exists an open set $ U_\lambda \subset \C^n $ on which $ f^\lambda $ vanishes, then $ f =0.$
	\end{thm}

\begin{rem} Note that when $ f $ is compactly supported the function $ f^\lambda $ is also compactly supported and hence vanishes on an open set. The same is true if we assume that $ f $ is supported on a cylindrical set $ \{ z \in \C^n: |z| < a \} \times \R.$ As $ \widehat{f}(\lambda) = \pi_\lambda(f^\lambda) $, the above can be considered as a result for the Weyl transform of functions on $ \C^n.$
\end{rem}

	Theorem \ref{ingh} was proved in \cite{I} by Ingham by making use of  Denjoy-Carleman theorem on quasi-analytic functions. In \cite{BRS}, the authors have used Radon transform and a several variable extension of Denjoy-Carleman theorem due to Bochner and Taylor \cite{BT} in order to prove the $n$-dimensional version of Theorem \ref{ingh}. An $ L^2 $  variant of the result of Bochner-Taylor which was proved by Chernoff in \cite{CH2} has turned out to be very useful in establishing Ingham type theorems.  
	
	\begin{thm} \cite[Chernoff]{CH2}
		\label{ch-euc}
		Let $f$ be a smooth function on $\mathbb{R}^n.$ Assume that $\Delta^mf\in L^2(\mathbb{R}^n)$ for all $m\in \mathbb{N}$ and that $\sum_{m=1}^{\infty}\|\Delta_{\mathbb{R}^n}^mf\|_2^{-\frac{1}{2m}}=\infty.$ If $f$ and all its partial derivatives   vanish at $0$, then $f$ is identically zero.
	\end{thm} 
As the Laplacian is translation invariant, $0$ can be replaced by any other point in the above theorem. As a matter of fact, this theorem shows how partial differential operators generate the class of quasi-analytic functions. Recently,  Bhowmik-Pusti-Ray \cite{BPR} have established an analogue of Chernoff's theorem for the Laplace-Beltrami operators on non-compact Riemannian symmetric spaces and use the same in proving a version of Ingham's theorem for the Helgason Fourier transform. 

In the context of the Heisenberg group, we prove Theorem \ref{ingh-hei-variant}, and hence Theorem \ref{ingh-hei}, by using   the following analogue of Chernoff's theorem for the family of special Hermite operators $L_{\lambda}.$ These operators on $ \C^n $  are defined via the relation  $\mathcal{L}(f(z)e^{i\lambda t})=e^{i\lambda t}L_{\lambda}f(z) $ where $ \mathcal{L} $ is the sublaplacian on $ \He^n.$ Observe that when $ \lambda = 0, $ the special Hermite operator $ L_\lambda $ reduces to $ \Delta $ on $ \C^n.$

\begin{thm}
	\label{ch-L-strong}
	For any fixed $ \lambda \in \R, $ let $f\in C^{\infty}(\mathbb{C}^n)$ be such that $L_{\lambda}^mf\in L^2(\mathbb{C}^n)$ for all $m\geq0$ and  that $\sum_{m=1}^{\infty}\|L_{\lambda}^mf\|_2^{-\frac{1}{2m}}=\infty.$ If  $f$ and all its partial derivatives vanish at some $w\in \mathbb{C}^n$, then $f$ is identically zero.
\end{thm}  When $ \lambda = 0 $, the above is just Chernoff's theorem for the Laplacian on $ \C^n $. For $\lambda=1$, a weaker version of the theorem, namely under the assumption that $ f $ vanishes on an open set,   has been proved in \cite[Theorem 4.1]{GT-adv}.  The weaker version is  in fact   good enough to prove Theorems \ref{ingh-hei-variant} and  \ref{ingh-hei}. However, in this paper, we prove the above  improvement which is the exact  analogue of Theorem \ref{ch-euc} for the special Hermite operators and  the second main result of this article.\\

	We conclude the introduction by briefly describing the organization of the paper. After recalling the required preliminaries regarding harmonic analysis on Heisenberg group  in Section 2, we prove an analogue of Chernoff's theorem for the special Hermite operators (Theorem \ref{ch-L-strong})  in Section 3. In section 4, we prove the Ingham's theorems on the Heisenberg group,  namely Theorems \ref{ingh-hei}, and \ref{ingh-hei-variant}.   
	
	\section{Preliminaries on Heisenberg groups}
	In this section, we collect the results which are necessary for the study of uncertainty principles for the Fourier transform on the Heisenberg group. We refer the reader to the  two classical books Folland \cite{F} and Taylor \cite{Taylor} for the preliminaries of harmonic analysis on the Heisenberg group. However, we will be closely following the notations of  the books of Thangavelu  \cite{TH2} and \cite{TH3}.  
	\subsection{ Heisenberg group and Fourier transform}
	Let $\mathbb{H}^n:=\mathbb{C}^n\times\mathbb{R}$ denote the $(2n+1)$-dimensional  Heisenberg group equipped with the group law 
	$$(z, t).(w, s):=\big(z+w, t+s+\frac{1}{2}\Im(z.\bar{w})\big),\ \forall (z,t),(w,s)\in \mathbb{H}^n.$$ This is a step two nilpotent Lie group where the Lebesgue measure $dzdt$ on $\mathbb{C}^n\times\mathbb{R}$ serves as the Haar measure. The representation theory of $\mathbb{H}^n$ is well-studied in the literature. In order to define Fourier transform, we use the Schr\"odinger representations as described below.  
	
	For each non-zero real number $ \lambda $, we have an infinite dimensional representation $ \pi_\lambda $ realised on the Hilbert space $ L^2( \R^n).$ These are explicitly given by
	$$ \pi_\lambda(z,t) \varphi(\xi) = e^{i\lambda t} e^{i\lambda(x \cdot \xi+ \frac{1}{2}x \cdot y)}\varphi(\xi+y),\,\,\,$$
	where $ z = x+iy $ and $ \varphi \in L^2(\R^n).$ These representations are known to be  unitary and irreducible. Moreover, by a theorem of Stone and Von-Neumann (see e.g., \cite{F}), these account, upto unitary equivalence, for all the infinite dimensional irreducible unitary representations of $ \mathbb{H}^n $ which act as $e^{i\lambda t}I,~\lambda\neq0$, on the center. Also, there is another class of one dimensional irreducible representations that corresponds to the case $\lambda=0.$ As they  do not contribute to the Plancherel measure  we will not describe them here.
	
	The Fourier transform of a function $ f \in L^1(\mathbb{H}^n) $ is the operator valued function obtained by integrating $ f $ against $ \pi_\lambda$:
	$$ \hat{f}(\lambda) = \int_{\mathbb{H}^n} f(z,t) \pi_\lambda(z,t)  dz dt .$$  Note that $ \hat{f}(\lambda) $ is a bounded linear operator on $ L^2(\R^n).$ 
	Now, by definition of $\pi_{\lambda}$ and $\hat{f}(\lambda)$, it is easy to see that 
	$$\widehat{f}(\lambda)=\int_{\C^n}f^{\lambda}(z)\pi_{\lambda}(z,0)dz, $$ where 
	$f^{\lambda}$ stands for the inverse Fourier transform of $f$ in the central variable:
	$$f^{\lambda}(z):=\int_{-\infty}^{\infty}e^{i\lambda t}f(z,t)dt.$$ This motivates the following definition. Given a function $g$ on $\C^n$, we consider the following   operator  defined by
	$$ W_{\lambda}(g):=\int_{\C^n}g(z)\pi_{\lambda}(z,0)dz.
	$$ With these notations, we note that  $\hat{f}(\lambda)=W_{\lambda}(f^{\lambda}).$  These transforms are called the Weyl transforms and for  $\lambda=1$, they are simply denoted by $ W(g) $ instead of $W_1(g).$ We have the following Plancherel formula for the Weyl transforms  (See \cite[2.2.9, Page no-49]{TH3})
	\begin{equation}
		\label{wplan}
		\|W_{\lambda}(g)\|^2_{HS}|\lambda|^n=(2\pi)^n\|g\|_2^2, ~g\in L^2(\mathbb{C}^n).
	\end{equation} 
	This, in view of the relation between the group Fourier transform and the Weyl transform, proves that when $ f \in L^1 \cap L^2(\mathbb{H}^n) $, its Fourier transform  is actually a Hilbert-Schmidt operator and one has
	$$ \int_{\mathbb{H}^n} |f(z,t)|^2 dz dt = (2\pi)^{-(n+1)}\int_{-\infty}^\infty \|\widehat{f}(\lambda)\|_{HS}^2 |\lambda|^n d\lambda,  $$
	where $\|.\|_{HS}$ denotes the Hilbert-Schmidt norm. 
	The above allows us to extend  the Fourier transform as a unitary operator between $ L^2(\mathbb{H}^n) $ and the Hilbert space of Hilbert-Schmidt operator valued functions  on $ \R $ which are square integrable with respect to the Plancherel measure  $ d\mu(\lambda) = (2\pi)^{-n-1} |\lambda|^n d\lambda.$ We polarize the above identity to obtain 
	$$\int_{\He^n}f(z,t)\overline{g(z,t)}dzdt=\int_{-\infty}^{\infty}tr(\widehat{f}(\lambda)\widehat{g}(\lambda)^*)~d\mu(\lambda).$$ Also for suitable functions $f$ on $\He^n$ we have the inversion formula
	$$f(z,t)=\int_{-\infty}^{\infty}tr(\pi_{\lambda}(z,t)^*\widehat{f}(\lambda))d\mu(\lambda).$$
	Moreover, the Fourier transform behaves well with the convolution of two functions defined by $$f\ast g(x):=\int_{\He^n}f(xy^{-1})g(y)dy.$$ In fact, for any $f,g\in L^1(\mathbb{H}^n)$, it follows from the definition that $$\widehat{f \ast g}(\lambda)=\hat{f}(\lambda)\hat{g}(\lambda).$$  In the following subsection, we describe the role of special functions in the harmonic analysis on $\mathbb{H}^n$ and show that the group Fourier transform of a suitable class of functions take a nice form. 
	\subsection{Special functions and Fourier transform}
	For each $\lambda\neq0$, we consider the following  family of scaled Hermite functions indexed by $\alpha\in\mathbb{N}^n$: $$\Phi_\alpha^{\lambda}(x):=|\lambda|^{\frac{n}{4}}\Phi_\alpha(\sqrt{|\lambda| }x),~x\in\mathbb{R}^n, $$ 
	where $\Phi_\alpha$ denote the $n-$dimensional Hermite functions (see \cite{TH1}). It is well-known that these  scaled functions $\Phi_\alpha^{\lambda}$ are eigenfunctions of the scaled Hermite operator $H(\lambda):=-\Delta_{\R^n}+\lambda^2|x|^2$ with eigenvalue $(2|\alpha|+n)|\lambda|$ and $\{\Phi_\alpha^{\lambda}:\alpha\in\mathbb{N}^n\}$ forms an orthonormal basis for $L^2(\R^n)$. As a consequence, 
	$$\|\widehat{f}(\lambda)\|_{HS}^2=\sum_{\alpha \in \mathbb{N}^n}\|\widehat{f}(\lambda)\Phi_\alpha^{\lambda}\|_2^2.$$ In view of this, the Plancheral formula takes the following very useful form 
	$$\int_{\mathbb{H}^n} |f(z,t)|^2 dz dt =  \int_{-\infty}^\infty \sum_{\alpha \in \mathbb{N}^n}\|\widehat{f}(\lambda)\Phi_\alpha^{\lambda}\|_2^2 \  d\mu( \lambda) . $$

	Given $\sigma\in U(n)$, we define $R_{\sigma}f(z,t)=f(\sigma.z,t)$. We say that a function $f$ on $\He^n$ is radial if it  is invariant under the action of $U(n)$ i.e., $R_{\sigma}f=f$ for all $\sigma\in U(n).$ The Fourier transforms of such radial integrable functions are  functions of the Hermite operator $ H(\lambda).$ 
	In fact, if  $ H(\lambda) = \sum_{k=0}^\infty (2k+n)|\lambda| P_k(\lambda)$ is the spectral decomposition of this operator, then for a radial intrgrable function $f$  we have
	$$ \widehat{f}(\lambda)  = \sum_{k=0}^\infty  R_k(\lambda, f) P_k(\lambda).$$ Here, $P_k(\lambda)$ stands for the orthogonal projection of $L^2(\mathbb{R}^n)$ onto the $k^{th}$ eigenspace spanned by scaled Hermite functions $\Phi^{\lambda}_{\alpha}$ with $|\alpha|=k$. The coefficients $ R_k(\lambda,f) $ are given by
	\begin{equation}
		R_k(\lambda,f)  =  \frac{k!(n-1)!}{(k+n-1)!} \int_{\C^n}  f^{\lambda }(z) \varphi^{n-1}_{k,\lambda}(z)~dz.
	\end{equation}
In the above formula, $ \varphi_{k,\lambda}^{n-1} $ are the Laguerre functions of type $ (n-1)$:
	$$  \varphi^{n-1}_{k,\lambda}(z)  = L_k^{n-1}(\frac{1}{2}|\lambda||z|^2) e^{-\frac{1}{4}|\lambda||z|^2}, $$ where $L^{n-1}_k$ denotes the Laguerre polynomial of type $(n-1)$. For the purpose of estimating the Fourier transform we need good estimates for the  Laguerre functions $\varphi_{k,\lambda}^{n-1}.$ In order to get such estimates, we use the available sharp estiamtes of standard Laguerre functions as described below in more general context.

	For any $ \delta > -1 $, let $ L_k^\delta(r) $ denote the Laguerre polynomials of type $ \delta$.  The standard Laguerre functions are defined by 
	$$
	\mathcal{L}_k^{\delta}(r)=\Big(\frac{\Gamma(k+1)\Gamma(\delta+1)}{\Gamma(k+\delta+1)}\Big)^{\frac12}L_k^{\delta}(r)e^{-\frac12r}r^{\delta/2}
	$$
	which  form an orthonormal system in $L^2((0,\infty),dr)$. In terms of $\mathcal{L}_k^{\delta}(r),$ we have
	$$
	\varphi_k^{\delta}(r)=2^{\delta}\Big(\frac{\Gamma(k+1)\Gamma(\delta+1)}{\Gamma(k+\delta+1)}\Big)^{-\frac12}r^{-\delta}\mathcal{L}_k^{\delta}\Big(\frac12r^2\Big).
	$$
	Asymptotic properties of $\mathcal{L}_k^{\delta}(r)$ are well-known in the literature, see \cite[Lemma 1.5.3]{TH1}. The estimates in \cite[Lemma 1.5.3]{TH1}   are sharp, see \cite[Section 2]{M} and \cite[Section 7]{Mu}.
	For our convenience, we restate the result in terms of $ \varphi_{k,\lambda}^{n-1}(r).$
	
	\begin{lem} \label{lem:T}   Let $ \nu(k) = 2(2k+n) $ and  $C_{k,n} = \left(\frac{k!(n-1)!}{(k+n-1)!}\right)^{\frac12}.$ For  $\lambda\neq0, $ we have the estimates
		$$
		C_{k,n}\,\, |\varphi_{k,\lambda}^{n-1}(r)|\le C(r\sqrt{|\lambda|})^{-(n-1)}\begin{cases} (\frac{1}{2}\nu(k)r^2 |\lambda| )^{(n-1)/2}, &0\le r\le \frac{\sqrt{2}}{\sqrt{\nu(k)|\lambda|}}\\
			(\frac{1}{2}\nu(k)r^2 |\lambda|)^{-\frac14}, &\frac{\sqrt{2}}{\sqrt{\nu(k)|\lambda|}}\le r\le \frac{\sqrt{\nu(k)}}{ \sqrt{|\lambda|}}\\
			\nu(k)^{-\frac14}(\nu(k)^{\frac13}+|\nu(k)-\frac{1}{2} |\lambda| r^2|)^{-\frac14}, &\frac{\sqrt{\nu(k)}}{ \sqrt{|\lambda|}}\le r\le \frac{\sqrt{3\nu(k)}}{ \sqrt{|\lambda|}}\\
			e^{-\frac12\gamma r^2 |\lambda| }, & r\ge \frac{\sqrt{3\nu(k)}}{ \sqrt{|\lambda|}},
		\end{cases}
		$$ where $\gamma>0$ is a fixed constant and $C $  is independent of $k$ and $\lambda$. 
	\end{lem}
\subsection{The sublaplacian and special Hermite operators}
We let $ \mathfrak{h}_n $ stand for the Heisenberg Lie algebra consisting of left invariant vector fields on $ \mathbb{H}^n .$  A  basis for $ \mathfrak{h}_n $ is provided by the $ 2n+1 $ vector fields
$$ X_j = \frac{\partial}{\partial{x_j}}+\frac{1}{2} y_j \frac{\partial}{\partial t}, \,\,Y_j = \frac{\partial}{\partial{y_j}}-\frac{1}{2} x_j \frac{\partial}{\partial t}, \,\, j = 1,2,..., n ,~\text{and}~  T = \frac{\partial}{\partial t}.$$ These correspond to certain one parameter subgroups of $ \mathbb{H}^n.$ The sublaplacian on $\He^n$ is defined by $$\mathcal{L}:=-\sum_{j=1}^{\infty}(X_j^2+Y_j^2) $$ which can be explicitly calculated as
$$\mathcal{L}=-\Delta_{\C^n}-\frac{1}{4}|z|^2\frac{\partial^2}{\partial t^2}+N\frac{\partial}{\partial t},$$ where  $\Delta_{\C^n}$ stands for the Laplacian on $\C^n$ and $N$ is the rotation operator defined by 
$$N=\sum_{j=1}^{n}\left(x_j\frac{\partial}{\partial y_j}-y_j\frac{\partial}{\partial x_j}\right).$$ This is a sub-elliptic operator and homogeneous of degree $2$ with respect to the non-isotropic dilation given by $\delta_r(z,t)=(rz,r^2t).$ The sublaplacian is also invariant under  rotation i.e., $R_{\sigma}\circ \mathcal{L}=\mathcal{L}\circ R_{\sigma},~\sigma\in U(n).$ For each $\lambda\neq0$, special Hermite operator $L_{\lambda}$ is defined via the relation $$\mathcal{L}(e^{i\lambda t}f(z))=e^{i\lambda t}L_{\lambda}f(z).$$ Furthermore, it is not hard to see that  $(\mathcal{L}f)^{\lambda}(z)=L_{\lambda}f^{\lambda}(z).$ It turns out that $ L_\lambda $  is explicitly given by 
$$L_\lambda = -\Delta_{\C^n}+\frac{1}{4} \lambda^2 |z|^2+ i \lambda N.$$ This family of special Hermite operators has a useful translation invariance property coming from the sublaplacian. 

Recall that the sublaplacian $\mathcal{L}$ is invariant under the left translations defined by $\tau_yf(x):=f(y^{-1}x),~x,y\in \He^n.$ In other words, $\tau_{y}(\mathcal{L}f)=\mathcal{L}(\tau_{y}f).$ Now, with $x=(w,0)\in\He^n,$  taking inverse Fourier transform in the central variable gives us
$$(\tau_{x}(\mathcal{L}f))^{\lambda}(z)=L_{\lambda}(\tau_{x}f)^{\lambda}(z)$$ which, after simplification leads to  $$ e^{\frac{i\lambda}{2}\Im(w.\bar{z})}L_{\lambda}f^{\lambda}(z-w)=L_{\lambda} (e^{\frac{i\lambda}{2}\Im(w.\bar{z})}f^{\lambda}(z-w)).$$ This observation in turn implies that the special Hermite operator $L_{\lambda}$ is invariant under the $\lambda$-twisted translation $T^{\lambda}_w, ~w\in \mathbb{C}^n$, defined by
\begin{equation}
	\label{twistedtrans}
	T^{\lambda}_wg(z):=e^{\frac{i\lambda}{2}\Im(w.\bar{z})}g(z-w).
\end{equation}  In other words, \begin{equation}
	\label{twistedinv}
	T^{\lambda}_w(L_{\lambda}g)=L_{\lambda}(T^{\lambda}_wg),~ w\in \mathbb{C}^n.
\end{equation}  It is also known that these $L_{\lambda}$'s are elliptic operators on $ \C^n $ with an explicit spectral decomposition. The spectrum consists of the real numbers of the form $ (2k+n)|\lambda|$,
$k \geq0 $, and the eigenspaces associated to each of these eigenvalues are infinite dimensional. 

In the following, we describe the spectral decomposition for the case when $\lambda=1$. For the sake of simplicity, we write  $L$ instead of $ L_1 $ and $ f\times g $ instead of $ f \ast_1 g.$ Thus, 
$$ f \times g(z) = \int_{\C^n} f(z-w) g(w) e^{\frac{i}{2} \Im (z \cdot \bar{w})} dw.$$
 It is known that (\cite[page no. 58]{TH3}) the special Hermite expansion of  a function $ f \in L^2(\C^n) $ and Parseval's identity reads as
\begin{equation}\label{parseval} f(z)  = (2\pi)^{-n} \sum_{k=0}^\infty f\times \varphi_k^{n-1}(z),\, \,\,  \|f\|_2^2  = (2\pi)^{-n} \sum_{k=0}^\infty  \|f\times \varphi_k^{n-1}\|_2^2
\end{equation}
and each $ f \times \varphi_k^{n-1} $ is an eigenfunction of the operator $ L $ with eigenvalue $ (2k+n).$    Now if $g(z)=g_0(|z|)$ is a radial function on $\mathbb{C}^n$, then $L_{\lambda}g$ takes the form 
$L_{\lambda}g=L_{\lambda,n-1}g_0$ where $L_{\lambda,n-1}$ is the scaled Laguerre operator of type $(n-1)$ given by 
$$L_{\lambda,n-1}:=-\frac{d^2}{dr^2}- \frac{2n-1}{r} \frac{d}{dr} +\frac{1}{4} \lambda^2r^2.$$ In what follows, when $\lambda=1$, we simply denote the radial part of the  special Hermite operator $L_{1,n-1}$ by $L_{n-1}.$ Also, in order to prove Chernoff's theorem for the special Hermite operator, we need to use Laguerre operators of more general type and eigenfunction expansion associated with them. In the following subsection, we develop notations and record required results related to Laguerre expansion in this connection. 

\subsection{Laguerre expansion}To start with, we first recall the definition of Laguerre polynomials.  For any $\delta\geq -\frac{1}{2},$ the Laguerre polynomials of type $\delta$ are defined by 
$$e^{-t}t^{\delta}L_k^{\delta}(t)=\frac{1}{k!}\frac{d^k}{dt^k}(e^{-t}t^{k+\delta})$$
for $t>0$, and $k \geq0.$ The explicit form of $L_k^{\delta}(t)$ which is a polynomial of degree $k$, is given by 
$$L_k^{\delta}(t)=\sum_{j=0}^{k}\frac{\Gamma(k+\delta+1)}{\Gamma(j+\delta+1)\Gamma(k-j+1)}\frac{(-t)^j}{j!}.$$  
We now introduce the normalised Laguerre functions $ \mathcal{L}_k^\delta $ defined as follows.
$$\mathcal{L}^{\delta}_k(t)=\left(\frac{\Gamma(k+1)}{\Gamma(k+1+\delta)}\right)^{\frac{1}{2}}e^{-\frac{t}{2}}t^{\frac{\delta}{2}}L^{\delta}_k(t),\:\, t >0.$$
Then it is well-known that for any fixed $\delta\geq -\frac{1}{2}$, $\left\lbrace\mathcal{L}^{\delta}_k\right\rbrace_{k=0}^{\infty}$ is an orthonormal basis for  $L^2(\mathbb{R}^+,dt).$ Now, fix $\delta\geq -\frac{1}{2}$, and consider the following Laguerre functions of type $\delta$ defined by 
$$\psi_k^{\delta}(r):= \frac{\Gamma(k+1)\Gamma(\delta)}{\Gamma(k+\delta+1)}L^{\delta}_k(\frac12r^2)e^{-\frac14r^2},~r>0. $$
It turns out that $\psi_k^{\delta}(0)=1$, and these are eigenfunctions of the following Laguerre operator of type $\delta$ given by 
$$L_{\delta}:=-\frac{d^2}{dr^2}-\frac{2\delta+1}{r}\frac{d}{dr}+\frac{1}{4}r^2$$
with eigenvalue $(2k+\delta+1)$ i.e., $L_{\delta}\psi_k^{\delta}=(2k+\delta+1)\psi_k^{\delta}.$ This can be checked using the relations \cite[1.1.48, 1.1.49]{TH1} satisfied by the Laguerre polynomials. We will see later that for $\delta=n-1$, $L_{\delta}$ corresponds to the radial part of the special Hermite operator. Now, using the orthogonality property of the functions $\mathcal{L}^{\delta}_k$ (mentioned above), it is not difficult to see that $\{\psi_k^{\delta}:k\geq 0\}$ forms an orthogonal basis for $L^2(\mathbb{R}^+, r^{2\delta+1}dr).$ In view of this, for $f\in L^2(\mathbb{R}^+,r^{2\delta+1}dr)$ we have
\begin{equation}
	\label{planlag}
	f(r)  = \sum_{k=0}^\infty   c^{\delta}_k \, \mathcal{R}^{\delta}_k(f) \psi^{\delta}_k(r),\,\,\, \|f\|_2^2=\sum_{k=0}^{\infty}c_k^{\delta}\, |\mathcal{R}^{\delta}_k(f)|^2,
\end{equation}  
where $(c^{\delta}_k)^{-1}:=\int_0^\infty |\psi^{\delta}_k(r)|^2 r^{2\alpha+1} dr$, and $\mathcal{R}^{\delta}_k(f)$ denotes the Laguerre coefficients of $f$ given by  $$\mathcal{R}^{\delta}_k(f) =  \int_0^\infty f(r) \psi^{\delta}_k(r) r^{2\delta+1} dr,~k\geq 0.$$ We have the following Chernoff type theorem for $L_{\delta}$: 
\begin{thm}\label{chernoff-laguerre}Let $\delta\geq -\frac{1}{2}$ and $ f \in L^2(\R^+, r^{2\delta+1}dr) $ be such that $ L_{\delta}^mf \in L^2(\R^+, r^{2\delta+1}dr) $ for all $ m \geq0 $,  and satisfies  the Carleman condition 
	$ \sum_{m=1}^\infty  \| L_{\delta}^m f \|_2^{-1/(2m)} = \infty.$ If $L_{\delta}^mf(0)=0$ for all $m\geq 0$, then $f$ is identically zero.
\end{thm}
For a proof of this result, we refer the reader to Theorem 2.4 and the Remark 2.5 after that in \cite{GT-adv}.

	\section{An analogue of Chernoff's theorem for the special Hermite operator}
	Our next aim is to prove  Theorem \ref{ch-L-strong}. For the sake of simplicity, we assume that $\lambda=1$ and prove the Theorem \ref{ch-L-strong} for $L.$  In proving the weaker version of Chernoff's theorem for $L$, in \cite{GT-adv}, the authors used twisted spherical means and a Chernoff type theorem for its radial part which is a Laguerre operator of type $(n-1).$ However, in this case, we have to consider Laguerre operators of a more general type, as well as the eigenfunction expansion that goes with them, which has already been described at end of the previous section. Furthermore, we will use Hecke-Bochner type identity for special Hermite projections, which requires some preparations. To begin with, closely following the notations of  \cite[Section 5, Chapter 2]{TH3}  we  describe bi-graded spherical harmonics on $\mathbb{C}^n.$ \\
	
 \textbf{Bi-graded spherical harmonics:} Let $p$ and $q$ be two non-negative integers. Suppose $\mathcal{P}_{p,q}$ denotes the set of all polynomials in $z$ and $\bar{z}$ of the form 
	$$P(z)=\sum_{|\alpha|\leq p,~ |\beta|\leq q}c_{\alpha,\beta}z^{\alpha}\bar{z}^{\beta}$$ which clearly has the following homogeneity property: $P(\lambda z)=\lambda^p\bar{\lambda}^qP(z),~\lambda\in \mathbb{C}.$ Now, in terms of the vector fields $\frac{\partial}{\partial z_j},~\frac{\partial}{\partial \bar{z}_j},~j=1,2,..,n$, the Laplacian on $\mathbb{C}^n$ has the form $\Delta_{\mathbb{C}^n}=4\sum_{j=1}^n\frac{\partial^2}{\partial z_j\partial \bar{z}_j}.$ In view of this, it can be checked that $\Delta_{\mathbb{C}^n}: \mathcal{P}_{p,q}\rightarrow \mathcal{P}_{p-1,q-1}$. We denote the kernel of $\Delta_{\mathbb{C}^n}$ by $\mathcal{H}_{p,q}.$ More precisely, 
	$$\mathcal{H}_{p,q}:=\{P\in \mathcal{P}_{p,q}: \Delta_{\mathbb{C}^n}P=0\},$$ which is called the set of all bi-graded solid harmonics of degree $(p,q).$ We define $$\mathcal{S}_{p,q}:=\{P|_{\mathbb{S}^{2n-1}}: P\in \mathcal{H}_{p,q}\}.$$ The elements of $\mathcal{S}_{p,q}$ are called the bi-graded spherical harmonics of degree $(p,q)$. This turns out be a Hilbert space under the usual inner-product of $L^2(\mathbb{S}^{2n-1}).$ Let $d(p,q)$ denote the dimension of this Hilbert space. Now, it is well-known that we can choose an orthonormal basis $\mathcal{B}_{p,q}:=\{S^j_{p,q}:1\leq j\leq d(p,q)\}$ for $\mathcal{S}_{p,q}$, for each pair of non-negative integers $(p,q)$ such that $\mathcal{B}:=\cup_{p,q\geq 0}\mathcal{B}_{p,q}$ forms an orthonormal basis for $L^2(\mathbb{S}^{2n-1}).$ For our purpose, we require the following Hecke-Bochner type identity in the context of special Hermite projections. 
	\begin{thm}
		\label{hb-spl-h}
		Suppose $f\in L^1(\mathbb{C}^n)$ has the form $f=Pg$ where $g$ is radial and $P\in\mathcal{H}_{p,q}$ for some $p,q\geq 0.$ Then $f\times \varphi_{k}^{n-1} = 0 $ unless $ k \geq p,$ in which case
		$$f\times \varphi_{k}^{n-1}(z) =(2\pi)^{-n}g\times \varphi_{k-p}^{n+p+q-1}(z)P(z),$$
		where the twisted convolution on the right hand side is on $\mathbb{C}^{n+p+q}.$
	\end{thm}

	For a proof of this result, we refer the reader to \cite[Theorem 2.6.1]{TH3}. We are now in a position to  prove the Theorem \ref{ch-L-strong}.\\

	\textbf{\textit{Proof of Theorem \ref{ch-L-strong}}:}  Let $f$ be as in the statement. The main idea is to reduce the matters to radial case by expanding $f$ in terms of bi-graded spherical harmonics and then use Chernoff's theorem for Laguerre operator of suitable type. The proof will be completed in the following steps.\\
	
		\textit{Step 1:}(Reduction of vanishing condition)
		Suppose $f$ and all its partial derivatives vanish at a point $0\neq w\in \mathbb{C}^n.$ Consider the function $g$ defined by $g=T^1_{-w}f$, which is nothing but the twisted translation of $f$ by $-w$ (See \eqref{twistedtrans}). In the following, we will be using standard multi-index notations. Using the product rule of partial derivatives, an easy calculation shows that $\partial^{\alpha}g(z)$ is equal to
		\begin{align*}
		\partial^{\alpha}(e^{-\frac{i}{2}\Im(w.\bar{z})}f(z+w))&=\sum_{\beta\leq \alpha}\binom{\alpha}{\beta}\partial^{\beta}(e^{-\frac{i}{2}\Im(w.\bar{z})})\partial^{\alpha-\beta}(f(z+w))\\&=\sum_{\beta\leq \alpha}\binom{\alpha}{\beta}(e^{-\frac{i}{2}\Im(w.\bar{z})})P_{\beta}(w,\bar{w})\partial^{\alpha-\beta}(f(z+w)),	
		\end{align*}
		 where $P_{\beta}(w,\bar{w})$ is some polynomial in $w$ and $\bar{w}$ whose explicit form is not required for our purpose. Note that for any multi-index $\alpha$, we have from the equation above  
		$$\partial^{\alpha}g(0)=\sum_{\beta\leq \alpha}\binom{\alpha}{\beta}P_{\beta}(w,\bar{w})\partial^{\alpha-\beta}f(w)=0$$ by the the assumption that $\partial^{\alpha}f(w)=0$ for all $\alpha.$ Furthermore, using the twisted translation invariance of $L$ (See \eqref{twistedinv}), it is not hard to see that $\|L^mg\|_{2}=\|L^mf\|_2$, whence $\|L^mg\|_2$ also satisfy the Carleman condition. Therefore, if $f$ and all its partial derivatives vanish at any point, we can simply work with a suitable twisted translate of $f.$ So, there is no loss of generality in assuming that $f$ and all its partial derivatives vanish at $0.$\\
		
		\textit{Step 2:} (Spherical harmonic coefficients of $L^mf$) The spherical harmonic expansion of $f$ reads as 
		$$f(z)=\sum_{p,q= 0}^{\infty}\sum_{j=1}^{d(p,q)}(f(r.),S^j_{p,q})_{L^2(\mathbb{S}^{2n-1})}S^j_{p,q}(\omega),~ z=r\omega.$$ Writing $f^j_{p,q}(r)=r^{-p-q} (f(r.),S^j_{p,q})_{L^2(\mathbb{S}^{2n-1})}$, and $P^j_{p,q}(z)=|z|^{p+q}S^j_{p,q}(\omega)$, we observe from the above that
		$$L^mf(z)=\sum_{p,q=0}^{\infty}\sum_{j=1}^{d(p,q)}L^m(f^j_{p,q}P^j_{p,q})(z) = \sum_{p,q= 0}^{\infty}\sum_{j=1}^{d(p,q)}L^mF_{p,q}^j(z),$$ where we have written $F_{p,q}^j(z):=f^j_{p,q}(|z|)P^j_{p,q}(z).$ Let us  calculate the special Hermite projections of $F_{p,q}^j$. In view of the Theorem \ref{hb-spl-h}, we see that for $ k\geq p,$
		\begin{align*}
			F_{p,q}^j\times \varphi_{k}^{n-1}(z)&=(2\pi)^{-n} P^j_{p,q}(z) \big(f^j_{p,q}\times \varphi_{k-p}^{n+p+q-1}(z)\big)\\
			&=(2\pi)^{-n} P^j_{p,q}(z)\  \mathcal{R}^{\delta(p,q)}_{k-p}(f^j_{p,q})\ \varphi_{k-p}^{\delta(p,q)}(z)
		\end{align*}
		where $\delta(p,q):=n+p+q-1.$ In the last equality, we have used the fact that $f^j_{p,q}$ can be thought of as a radial function on $\mathbb{C}^{n+p+q}.$ Therefore, we obtain from the special Hermite expansion of $F$ that 
		\begin{align}
			L^mF_{p,q}^j(z)&=(2\pi)^{-n}\sum_{k=0}^{\infty}(2k+n)^mF_{p,q}^j\times \varphi_{k}^{n-1}(z)\nonumber\\
			&=(2\pi)^{-2n} P^j_{p,q}(z)\ \sum_{k=p}^{\infty}\ (2k+n)^m\  \mathcal{R}^{\delta(p,q)}_{k-p}(f^j_{p,q})\ \varphi_{k-p}^{\delta(p,q)}(z)\nonumber\\
			&=(2\pi)^{-2n}\ P^j_{p,q}(z)\ \sum_{k=0}^{\infty}(2k+2p+n)^m\ \mathcal{R}^{\delta(p,q)}_{k}(f^j_{p,q})\ \varphi_{k}^{\delta(p,q)}(z)\nonumber\\
			&=(2\pi)^{-2n}\ S^j_{p,q}(\omega)\,r^{p+q}\ \sum_{k=0}^{\infty}(2k+2p+n)^m\ \mathcal{R}^{\delta(p,q)}_{k}(f^j_{p,q})\ \varphi_{k}^{\delta(p,q)}(r),\,\,\,~z=r\omega.
		\end{align}
		Thus, for a fixed $ m $  the spherical harmonic coefficients of $ L^mf(r\cdot) $ are given by 
		\begin{equation}
			\label{splheq1}
			G_{p,q}^j(r):=(L^mf(r.), S^j_{p,q})_{L^2(\mathbb{S}^{2n-1})}=(2\pi)^{-2n}\ r^{p+q}\ \sum_{k=0}^{\infty}(2k+2p+n)^m\ \mathcal{R}^{\delta(p,q)}_{k}(f^j_{p,q})\ \varphi_{k}^{\delta(p,q)}(r)
		\end{equation}
		for any $ r > 0.$  Now, by using the orthogonality of the Laguerre functions, we get from \eqref{splheq1} that
		\begin{align}
			\label{splheq2}
			\int_{0}^{\infty}&G_{p,q}^j(r)\varphi_k^{\delta(p,q)}(r)r^{2n+p+q-1}dr\nonumber\\&=(2\pi)^{-2n}(2k+2p+n)^m\, \mathcal{R}^{\delta(p,q)}_{k}(f^j_{p,q})\, \|\varphi_{k}^{\delta(p,q)}\|_2^2\nonumber\\
			&=c_n\, (2k+2p+n)^m\frac{k!(n+p+q-1)!}{(k+n+p+q-1)!}\mathcal{R}^{\delta(p,q)}_{k}(f^j_{p,q}).
		\end{align}\\

		\textit{Step 3:}(Carleman condition)
		We consider the function $f^j_{p,q}$ for fixed $j,p$ and $q$. With $\delta(p,q)$ as above, in view of the Plancherel formula \eqref{planlag}, we see that for any $m\geq 1$,
		\begin{align}
			\label{splheq3}
			\|L^m_{\delta(p,q)}f^j_{p,q}\|_2^2&=\sum_{k=0}^{\infty}(2k+\delta(p,q)+1)^{2m}\, c^{\delta(p,q)}_k \, |\mathcal{R}^{\delta(p,q)}_{k}(f^j_{p,q})|^2\nonumber\\
			&=\sum_{k=0}^{\infty} C(k,m,p,q,n) \left|\int_{0}^{\infty}G_{p,q}^j(r)\, \varphi_k^{\delta(p,q)}(r)\, r^{2n+p+q-1}dr\right|^2,
		\end{align}
		where we have used \eqref{splheq2}.  Here, $C(k,m,p,q,n)$ is given by 
		\begin{align*}
			C(k,m,p,q,n):=\left(\frac{2k+\delta(p,q)+1}{2k+2p+1}\right)^{2m}\,  c^{\delta(p,q)}_k \, \left(\frac{(k+n+p+q-1)!}{k!(n+p+q-1)!}\right)^2.
		\end{align*}
		Now, using the value of $\delta(p,q)$, we see that $$\frac{2k+\delta(p,q)+1}{2k+2p+1}\leq 1+\frac{q}{p}:=a_{p,q}.$$ Using this, we have from \eqref{splheq3} that
		\begin{align}
			\|L^m_{\delta(p,q)}f^j_{p,q}\|_2^2&\leq a_{p,q}^{2m}\sum_{k=0}^{\infty}c_k^{\delta(p,q)}\left|\int_{0}^{\infty}r^{-p-q}G_{p,q}^j(r)\, \psi_k^{\delta(p,q)}(r)\, r^{2n+2p+2q-1}dr\right|^2.
		\end{align}
		We observe that the expression inside the modulus sign on the right hand side is the Laguerre coefficient $\mathcal{R}^{\delta(p,q)}_k(r^{-p-q}G_{p,q}^j).$ Therefore,  by the Plancherel formula \eqref{planlag}, we obtain 
		\begin{align*}
			\|L^m_{\delta(p,q)}f^j_{p,q}\|_2^2&\leq a_{p,q}^{2m} \int_{0}^{\infty}|r^{-p-q}G_{p,q}^j(r)|^2r^{2n+2p+2q-1}dr\\
			&= a_{p,q}^{2m} \int_{0}^{\infty} |G_{p,q}^j(r)|^2\, r^{2n-1}dr.
		\end{align*}
		Recalling the definition of $ G_{p,q}^j(r) $, we observe  that 
		$$|G_{p,q}^j(r)|=|(L^mf(r.), S^j_{p,q})_{L^2(\mathbb{S}^{2n-1})}|\leq \|L^mf(r.)\|_{L^2(\mathbb{S}^{2n-1})}.$$
		Using this in the equation above and integrating in polar coordinates, we obtain
		\begin{equation}
			\|L^m_{\delta(p,q)}f^j_{p,q}\|_2^2\leq a_{p,q}^{2m}\|L^mf\|_2^2.
		\end{equation}
		Thus, the Carleman condition on $L^mf$ implies the Carleman condition 
		\begin{equation}
			\label{carlemanlagg}
			\sum_{m=1}^{\infty}\|L^m_{\delta(p,q)}f^j_{p,q}\|^{-1/(2m)}_2=\infty
		\end{equation} 
		for any spherical harmonic coefficient $ f_{p,q}^j.$ \\
		
		\textit{Step 4:}(Vanishing condition) We have assumed that $f$ and all its partial derivatives vanish at the origin. However, for our purpose, it is more convenient to work with the following equivalent vanishing condition written in terms of polar coordinates:
		\begin{equation}
			\label{polar-vanish}
			\left(\frac{d}{dr}\right)^mf(r\omega)|_{r=0}=0,\:\:\:\:\text{for all}\:\: \omega\in \mathbb{S}^{2n-1},\:  m\geq0.
		\end{equation} Indeed, it can be checked that
		$$  \left(\frac{d}{dr}\right)^kf(r\omega)  = \sum_{|\alpha|=k} \partial^\alpha f(r\omega)\,  \omega^\alpha. $$
		Hence, $(\frac{d}{dr})^kf(r\omega)|_{r=0} = 0 $, for all $ k $ if and only if $ \partial^\alpha f(0) = 0 $, for all $ \alpha.$ 
		We recall that $f^{j}_{p,q}$ is explicitly given by 
		$$f^j_{p,q}(r)=r^{-p-q}\int_{\mathbb{S}^{2n-1}}f(r\omega)S^j_{p,q}(\omega)d\sigma(\omega).$$ In view of the vanishing condition \eqref{polar-vanish},
	 a calculation using repeated application of L'Hospital rule, we verify  that all the derivatives of $f^j_{p,q}$ at $0$ are zero. Thus, $L_{\delta(p,q)}^mf^j_{p,q}(0)=0$, for all $m\geq 0$. Hence, by Chernoff's theorem for $L_{\delta(p,q)}$ (See Theorem \ref{chernoff-laguerre}), we have $f^j_{p,q}=0$, for all $j,p,q$. Therefore, we conclude that $f=0,$ thereby completing the proof.\qed

	\section{Ingham's theorem on the Heisenberg group}  
	
In this section we prove  Theorems \ref{ingh-hei}, and \ref{ingh-hei-variant} using Chernoff's theorem for the special Hermite operator. We first show the existence of a compactly supported function $ f $ on $ \He^n $ whose Fourier transform has a prescribed decay as stated in  Theorem \ref{ingh-hei}. This proves the sufficiency part of the condition on the function $ \Theta $ appearing in the hypothesis. We then use this part of the  theorem to prove the necessity of the condition on $ \Theta.$ We begin with some preparations.
	
	\subsection{Construction of $ F $} The Koranyi norm of  $x=(z,t)\in\mathbb{H}^n$,  is defined  by $|x|=|(z,t)|=(|z|^4+t^2)^{\frac{1}{4}}.$ In what follows, we  work with the following left invariant metric defined by  $ d(x,y):=|x^{-1}y|,\ x,y\in\mathbb{H}^n.$ Given $a\in\mathbb{H}^n$ and $r>0$, the open ball of radius $r$ with centre at $a$ is defined by 
	$$  B(a,r):=\{x\in\mathbb{H}^n:|a^{-1}x|<r\}.$$ 
	With this definition, we note that if $f,g:\mathbb{H}^n\rightarrow\mathbb{C}$ are  such that $\text{supp}(f)\subset B(0,r_1)$ and $\text{supp}(g)\subset B(0,r_2)$, then we have 
	$$   \text{supp}(f\ast g)\subset B(0,r_1).B(0,r_2)\subset B(0,r_1+r_2) , $$
	where  $f\ast g(x)=\int_{\mathbb{H}^n}f(xy^{-1})g(y)dy $ is the convolution of $ f $ with $ g.$

	Suppose $\{\rho_j\}_j$ and $\{\tau_j\}_j$ are two sequences of positive real numbers such that both the series $\sum_{j=1}^\infty \rho_j$ and $\sum_{j=1}^\infty \tau_j$ are convergent.  We let $ B_{\C^n}(0,r) $ stand for the ball of radius $ r $ centered at $ 0 $ in $ \C^n$ and let $\chi_S$ denote the characteristic function of a set $S.$ For each $j\in\mathbb{N},$ we define 
	functions $ f_j $ on $ \C^n $ and $ g_j $ on $ \R $ by
	$$f_j(z):=\rho_j^{-2n}\chi_{B_{\mathbb{C}^n}(0,a \rho_j)}(z),\,\,\,  ~z\in\C^n;$$ $$g_j(t):=\tau_j^{-2}\chi_{[-\tau_j^2/2,\tau_j^2/2]}(t),~t\in\R,$$ where the positive constant $ a $ is chosen so that $\|f_j\|_{L^1(\mathbb{C}^n)}=1.$ We now  consider the functions $F_j:\mathbb{H}^n\rightarrow\mathbb{C}$ defined by 
	$$F_j(z,t):=f_j(z)g_j(t),\ (z,t)\in\mathbb{H}^n.$$
	In the following lemma, we record some useful, but easily proven properties of these functions.
	\begin{lem}
		Let $F_j$ be as above and  define $ G_N = F_1 \ast F_2\ast.....\ast F_N.$ Then we have  
		\begin{enumerate}
			\item $\|F_j\|_{L^{\infty}(\mathbb{H}^n)}\leq \rho_j^{-2n}\tau_j^{-2},$\,\, \,  $\|F_j\|_{L^1(\mathbb{H}^n)}=1,$
			\item $\text{supp}(F_j)\subset B_{\mathbb{C}^n}(0,a \rho_j)\times [ -\tau_j^2/2,\tau_j^2/2]\subset B(0, a \rho_j+ c \tau_j) $, where $ 4 c^4 =1.$
			\item  For any $N\in\mathbb{N}$, $\text{supp} (G_N) \subset B(0, a \sum_{j=1}^{N}\rho_j+ c  \sum_{j=1}^{N}\tau_j),\,\, \|G_N\|_1 =1.$
			\item Given $x\in\mathbb{H}^n$, and $N\in\mathbb{N}$, $ F_2 \ast F_3 .....\ast F_N(x) \leq \rho^{-2n}_2\tau_2^{-2}.$ 
		\end{enumerate} 
	\end{lem}
	We also recall a result about Hausd\"orff measure which will be used in the proof of the next theorem.  Let $\mathcal{H}^n(A)$ denote the $n$-dimensional Hausdorff measure of 
	$A\subset \mathbb{R}^n.$ Hausd\"orff measure coincides with the Lebesgue measure for Lebesgue measurable sets. For sets in $\mathbb{R}^n$ with sufficiently nice boundaries, the $(n-1)$-dimensional Hausdorff measure is same as the surface measure. For more about this, we refer the reader to \cite[Chapter 7 ]{SS}. Let $ A \Delta B $ stands for the symmetric difference between any two sets $ A $ and $ B.$ See \cite{ DS} for a proof of the following theorem.
	\begin{thm}\label{hauss}
		Let $ A \subset \mathbb{R}^n$ be a bounded set. Then for any $ \xi \in \R^n $,
		$$\mathcal{H}^n(A\Delta(A+\xi))\leq |\xi| \mathcal{H}^{n-1}(\partial A), $$
		where $ A+\xi$ is the translation of $ A $ by $ \xi$ and $ \partial A $ is the boundary of $ A.$
	\end{thm}

	\begin{thm}\label{construct} The sequence defined by $ G_k  = F_1 \ast F_2\ast.....\ast F_k $ converges to a compactly supported  non-trivial function $ F \in L^2(\He^n)$ in $L^2(\He^n)$.
	\end{thm}
	\begin{proof}
		In order show that $ \{G_k \}$ is Cauchy in $ L^2(\He^n) $, we first estimate $\| G_{k+1}-G_k \|_{L^{\infty}(\mathbb{H}^n)}.$ As  all the functions $ F_j $ have unit $ L^1 $ norm, we have for any $x\in\mathbb{H}^n$ 
		\begin{align*}
			G_{k+1}(x)-G_k(x)&=\int_{\mathbb{H}^n} G_{k}(xy^{-1})F_{k+1}(y)dy-G_k(x) \int_{\mathbb{H}^n}F_{k+1}(y)dy  \\
			&=  \int_{\mathbb{H}^n}\left(G_k(xy^{-1})-G_k(x)\right)F_{k+1}(y)dy.
		\end{align*}   
		Since $ F_j $'s are even, we can change $ y $ into $ y^{-1}  $ in the above and estimate the same as
		\begin{equation}
			\label{eq1}
			|G_{k+1}(x)-G_{k}(x)|\leq \int_{\mathbb{H}^n}\left|G_k(xy)-G_k(x)\right|F_{k+1}(y)dy.
		\end{equation} 
		By defining $ H_{k-1} = F_2 \ast F_3......\ast F_k $, we note that $ G_k = F_1 \ast H_{k-1}$. Thus,
		$$ G_k(xy)-G_k(y) = \int_{\mathbb{H}^n}\left(F_1(xyu^{-1})-F_1(xu^{-1})\right) H_{k-1}(u) du.$$
		Using the estimate (4) in Lemma 4.1, we get that
		\begin{equation}\label{eq2}
			|G_k(xy)-G_k(x)|\leq \rho_2^{-2n}\tau_2^{-2}\int_{\mathbb{H}^n}\left|F_1(xyu^{-1})-F_1(xu^{-1})\right|du.
		\end{equation}
		The change of variables $ u \rightarrow u x$  transforms the integral in the right hand side  of the inequality above into
		$$\int_{\mathbb{H}^n}\left|F_1(xyu^{-1})-F_1(xu^{-1})\right|du=\int_{\mathbb{H}^n}\left|F_1(xyx^{-1}u^{-1})-F_1( u^{-1})\right|du.$$  Since  the group $\mathbb{H}^n$ is unimodular, another change of variables $u \rightarrow u^{-1}$  yields
		$$\int_{\mathbb{H}^n}\left|F_1(xyx^{-1}u^{-1})-F_1( u^{-1})\right|du=\int_{\mathbb{H}^n}\left|F_1(xyx^{-1}u )-F_1( u )\right| du.$$

		Let $ x =(z,t) =(z,0)(0,t),~ y =(w,s)=(w,0)(0,s).$ As $ (0,t)$ and $ (0,s) $ belong to the center of $ \He^n $, an easy calculation shows that $ xyx^{-1} = (w,0) (0, s+\Im(z \cdot \bar{w})).$ With $ u=(\zeta,\tau) $ we have 
		$$ xyx^{-1}u = (w+\zeta,0)(0, \tau+ s+ \Im(z \cdot \bar{w})-(1/2)\Im(\zeta\cdot \bar{w})).$$
		Since $ F_1(z,t) = f_1(z)g_1(t) $, we see that the integrand $F_1(xyx^{-1}u )-F_1( u )$ in the above integral takes the form
		$$  f_1(w+\zeta) g_1(\tau+ s+ \Im(z \cdot \bar{w})-(1/2)\Im(\zeta\cdot \bar{w}))- f_1(\zeta)g_1(\tau).$$
		By setting $ b =  b(s, z,w,\zeta) =s+ \Im(z \cdot \bar{w})-(1/2) \Im(\zeta\cdot \bar{w}) $, we can rewrite the above as
		\begin{equation}\label{eq3} \big( f_1(w+\zeta)-f_1(\zeta)\big) g_1(\tau+b)+ f_1(\zeta)\big( g_1(\tau+ b)-g_1(\tau)\big).\end{equation}

		In order to estimate the contribution of the second term in (\ref{eq3}) to  the integral under consideration,
		we first estimate the integral in $ \tau$-variable as follows:
		$$  \int_{-\infty}^\infty | g_1(\tau+b)-g_1(\tau)|  d\tau =  \tau_1^{-2} | (-b+K_\tau) \Delta K_\tau |,$$
		where $ K_\tau = [-\frac{1}{2}\tau_1^2, \frac{1}{2}\tau^2] $  is the support of $ g_1.$  For $ \zeta $ in the support of $ f_1 ,$ we have $ |\zeta| \leq a\rho_1 $, and hence
		$$  | (-b+K_\tau) \Delta K_\tau | \leq 2 |b(s,z,w,\zeta)| \leq ( 2|s|+ |z||w|+ a\rho_1 |w|).   $$
		Thus, we have proved the following estimate 
		\begin{equation}\label{est-one}
			\int_{\He^n}  f_1(\zeta)  |  g_1(\tau+ b)-g_1(\tau)| d\zeta d\tau \leq C \big(2|s|+(a\rho_1+|z|)|w| \big).\end{equation}
		As the integral of $ g_1 $ is one, the contribution of the first term in (\ref{eq3}) is given by
		$$  \int_{\C^n} | f_1(w+\zeta)-f_1(\zeta)| d\zeta   =  \rho_1^{-2n} \mathcal{H}^{2n}\left((-w+ B_{\mathbb{C}^n}(0, a\rho_1))\Delta B_{\mathbb{C}^n}(0,a\rho_1) \right) .$$
		By appealing to Theorem \ref{hauss} in estimating the above, we obtain 
		\begin{equation}\label{est-two}
			\int_{\He^n} | f_1(w+\zeta)-f_1(\zeta)| \, g( \tau+b) \, d\zeta d\tau  \leq C |w| .
		\end{equation}
		Using the estimates (\ref{est-one}) and (\ref{est-two}) in (\ref{eq2}) we obtain
		$$ |G_k(xy)-G_k(x)|\leq  C \rho_2^{-2n}\tau_2^{-2} \big( |s|+ (c_1+ c_2|z|)|w|)\big).$$ 
		This estimate, when used in (\ref{eq1}), in turn gives us 
		\begin{equation}\label{est-three}
			|G_{k+1}(z,t)-G_{k}(z,t)|\leq C \int_{\mathbb{H}^n} \big( |s|+ (c_1+ c_2|z|)|w|)\big) F_{k+1}(w,s)\, dw\,ds
		\end{equation}
		where the constants $ c_1, c_2 $ and $ C $ depend only on $ n.$ Recalling that on the support of $ F_{k+1}(w,s) = f_{k+1}(w)g_{k+1}(s) $,  $ |w| \leq \rho_{k+1} $ and $ |s| \leq \tau_{k+1}^2 $, the above yields the estimate
		\begin{equation}\label{est-four}
			|G_{k+1}(z,t)-G_{k}(z,t)|\leq C \big( \tau_{k+1}^2 + (c_1+c_2 |z|) \rho_{k+1} \big).
		\end{equation}
		It is easily seen that the support of $ G_{k+1}-G_k $ is contained in $ B(0, a\rho+c\tau) $ where $ \rho = \sum_{j=1}^\infty \rho_j $ and $ \tau = \sum_{\tau_j}.$ Consequently, from the above we conclude that
		$$  \|G_{k+1}-G_k\|_2 \leq  \| G_{k+1} - G_k \|_{\infty} \big(| B(0, a\rho+c\tau)|\big)^{1/2}  \leq C \big( \tau_{k+1}^2 + c_3 \rho_{k+1} \big).$$
		From the above, it is clear that $ \{G_k\} $ is Cauchy in $ L^2(\He^n) $, and hence converges to a function $ F \in L^2(\He^n) $ whose support  is contained in $ B(0,a\rho+c\tau).$  The same argument shows that $ \{G_k\} $ converges to $ F $ in $ L^1.$ As $ \|G_k\|_1 =1 $ for any $ k,$ it follows that $ \|F||_1 =1 $ and hence $ F $ is nontrivial.
	\end{proof}
	
	\subsection{Estimating the Fourier transform of $ F$}
	Suppose now that $\Theta $ is an even, decreasing function on $ \R $ for which $ \int_1^\infty \Theta(t) t^{-1} dt < \infty.$ We want to  choose two sequences of positive real numbers $ \{\rho_j\} $ and $ \{\tau_j\} $  in terms of $ \Theta $ so that the series $ \sum_{j=1}^\infty \rho_j $ and $ \sum_{j=1}^\infty \tau_j $ both converge. We can then construct a function $ F $ as in Theorem \ref{construct} which will be compactly supported. Having done the construction we now want to compute the Fourier transform of the constructed function $ F $ and compare it with $ e^{-  \Theta(\sqrt{H(\lambda)}) \sqrt{H(\lambda)}}.$ This can be achieved by a judicious choice of the sequences $ \{\rho_j\} $ and $ \{\tau_j\}.$ As $ \Theta $ is given to be decreasing, it follows that $ \sum_{j=1}^\infty  \frac{\Theta(j)}{ j} < \infty.$ It is then possible to  choose a decreasing sequence $ \{\rho_j\} $ such that $ \rho_j \geq  c_n^2 e^2  \frac{\Theta(j)}{j} $ (for a constant $ c_n $ to be chosen later) and $ \sum_{j=1}^\infty \rho_j < \infty.$ Similarly, we choose another decreasing sequence $ \{\tau_j \}$ such that $ \sum_{j=1}^\infty \tau_j <\infty.$
	
	In the proof of the following lemma we require good estimates for the Laguerre coefficients of the function $ f_j(z) =\rho_j^{-2n}\chi_{B_{\mathbb{C}^n}(0,a \rho_j)}(z) $ where $ a $ chosen so that $ \|f_j\|_1 =1.$ These coefficients are defined by
	\begin{equation}\label{lag-co} R_k^{n-1}(\lambda, f_j) =  \frac{   k! (n-1)!}{(k+n-1)!}  \int_{\C^n}  f_j(z) \varphi_{k,\lambda}^{n-1}(z) dz .\end{equation}
	
	\begin{lem}\label{est-four-lem}  There exists a constant $ c_n > 0 $ such that 
		$$  |R_k^{n-1}(\lambda,f_j)| \leq c_n \big(\rho_j \sqrt{(2k+n)|\lambda|}\big)^{-n+1/2}.$$
	\end{lem}
	\begin{proof} By abuse of notation we write $ \varphi_{k,\lambda}^{n-1}(r)$ in place of $\varphi_{k,\lambda}^{n-1}(z)$ when $ |z| =r.$ As $ f_j $ is defined as the dilation of a radial function,  the Laguerre coefficients are given by the integral
		\begin{equation}\label{lag-co1} R_k^{n-1}(\lambda, f_j) =  \frac{2 \pi^n}{\Gamma(n)} \frac{   k! (n-1)!}{(k+n-1)!}  \int_0^a   \varphi_{k,\lambda}^{n-1}(\rho_j r) r^{2n-1} dr .\end{equation}  
		When $ a \leq (\rho_j \sqrt{ (2k+n)|\lambda| })^{-1} $ we use the bound $\frac{   k! (n-1)!}{(k+n-1)!} |\varphi_{k,\lambda}^{n-1}(r)| \leq 1$ to  estimate  
		$$          \frac{2 \pi^n}{\Gamma(n)}  \frac{   k! (n-1)!}{(k+n-1)!}  \int_0^a   \varphi_{k,\lambda}^{n-1}(\rho_j r) r^{2n-1} dr   \leq  \frac{ \pi^n a^{n+1/2}}{\Gamma(n+1)} 
		\big(\rho_j \sqrt{(2k+n)|\lambda|}\big)^{-n+1/2}.$$ 
		When $ a > (\rho_j \sqrt{ (2k+n)|\lambda| })^{-1} $ we split the integral into two parts, one of which gives the same estimate as above. To estimate the integral taken over $ (\rho_j \sqrt{ (2k+n)|\lambda| })^{-1} < r < a ,$ we use the bound  stated in Lemma  \ref{lem:T} which leads to the estimate 
		$$          \frac{2 \pi^n}{\Gamma(n)}  \frac{   k! (n-1)!}{(k+n-1)!}  \int_{(\rho_j \sqrt{ (2k+n)|\lambda| })^{-1}}^a   \varphi_{k,\lambda}^{n-1}(\rho_j r) r^{2n-1} dr $$  
		$$  \leq  C_n  \big(\rho_j \sqrt{(2k+n)|\lambda|}\big)^{-n+1/2} \int_0^a r^{n-1/2} dr = C_n^\prime a^{n+1/2}  \big(\rho_j \sqrt{(2k+n)|\lambda|}\big)^{-n+1/2}.$$
		Combining the two estimates we get the lemma.
	\end{proof}
	
	\begin{thm}\label{compute}
		Let $\Theta:\R\rightarrow [0,\infty)$ be an even, decreasing function with $\lim_{\lambda \rightarrow \infty}\Theta(\lambda)=0$  for which $\int_{1}^{\infty}\frac{\Theta(\lambda)}{\lambda} d\lambda <\infty.$ Let $ \rho_j $ and $ \tau_j $ be chosen as above. Then the Fourier transform of the function $ F $ constructed in Theorem \ref{construct} satisfies the estimate
		\begin{align*}
			\widehat{F}(\lambda)^*\widehat{F}(\lambda)\leq e^{-2\Theta (\sqrt{H(\lambda)}) \sqrt{H(\lambda)}}, \ \lambda\neq0.
		\end{align*}
		
	\end{thm}
	
	\begin{proof} Observe that  $ F $ is radial  since each $ F_j $ is radial and hence the Fourier transform $ \widehat{F}(\lambda) $ is a function of the Hermite opertaor $ H(\lambda).$ More precisely, 
		\begin{align}
			\widehat{F}(\lambda)= \sum_{k=0}^{\infty}R_k^{n-1}(\lambda, F) P_k(\lambda)
		\end{align}
		where  the Laguerre coefficients are  explicitly given by  (see  (2.4.7) in \cite{TH3}. There is a typo- the factor $ |\lambda|^{n/2}$ should not be there)
		$$ R_k^{n-1}(\lambda, F) =  \frac{   k! (n-1)!}{(k+n-1)!}  \int_{\C^n}  F^\lambda(z) \varphi_{k,\lambda}^{n-1}(z) dz .$$
		In the above, $ F^\lambda(z) $ stands for the inverse Fourier transform of $ F(z,t) $ in the $ t $ variable.
		Expanding any $ \varphi \in L^2(\R^n) $ in terms of $ \Phi_\alpha^\lambda$ it is easy to see that the conclusion $\widehat{F}(\lambda)^*\widehat{F}(\lambda)\leq e^{-2\Theta (\sqrt{H(\lambda)}) \sqrt{H(\lambda)}} $ follows once we show that 
		$$ (R_k^{n-1}(\lambda, F))^2 \leq C e^{-2\Theta (\sqrt{(2k+n)|\lambda}) \sqrt{(2k+n)|\lambda|}}   $$
		for all $ k \in \mathbb{N} $ and $ \lambda \in \R^\ast.$ Now note that, by definition of $g_j$ and the choice of $a,$ we have 
		$$|\widehat{g}_j(\lambda)|=\left|\frac{\sin (\frac{1}{2}\tau_j^2\lambda)}{\frac{1}{2}\tau_j^2\lambda}\right|\leq 1, \,\,\,  |R_k^{n-1}(\lambda, f_j)| \leq 1.$$ 
		The bound on $ R_k^{n-1}(\lambda, f_j) $  follows from the fact that  $ |\varphi_k^\lambda(z)| \leq  \frac{(k+n-1)!}{k!(n-1)!} .$ Since $ F $ is constructed as the $ L^2 $ limit of the $ N$-fold convolution $ G_N = F_1 \ast F_2 ......\ast F_N $ we observe that for any $ N $
		$$ ( R_k^{n-1}(\lambda, F))^2 \leq (R_k^{n-1}(\lambda, G_N))^2   =  (  \Pi_{j=1}^N R_k^{n-1}(\lambda, F_j))^2 $$
		and hence it is enough to show that  for a given $ k $ and $ \lambda$ one can choose $ N=N(k,\lambda) $ in such a way that 
		\begin{equation}\label{eq6} (  \Pi_{j=1}^N R_k^{n-1}(\lambda, F_j))^2 \leq C e^{-2\Theta (\sqrt{(2k+n)|\lambda|}) \sqrt{(2k+n)|\lambda|}} .\end{equation}
		where $ C $ is independent of $ N.$   From the definition of $ G_N $ it follows that 
		$$ \widehat{G_N}(\lambda) = \Pi_{j=1}^N \widehat{F_j}(\lambda)  =   \Pi_{j=1}^N \big( \sum_{k=0}^\infty  R_k^{n-1}(\lambda, F_j) P_k(\lambda) \big)$$
		and hence  $ R_k^{n-1}(\lambda, G_N) =  \Pi_{j=1}^N R_k^{n-1}(\lambda, F_j).$ As $ F_j(z,t)  = f_j(z) g_j(t) $, we have
		$$   R_k^{n-1}(\lambda, G_N) = \big( \Pi_{j=1}^N  \widehat{g_j}(\lambda) \big)  \big( \Pi_{j=1}^N R_k^{n-1}(\lambda, f_j) \big)  .                   $$
		As the first factor  is bounded by one, it is enough to consider the product $\Pi_{j=1}^N R_k^{n-1}(\lambda, f_j).$
		
		We now  choose $ \rho_j $ satisfying  $ \rho_j \geq  c_n^2\, e^2  \frac{\Theta(j)}{j} $,  where $ c_n $ is the same constant appearing in Lemma \ref{est-four-lem}. We then take
		$ N=  \lfloor \Theta(((2k+n)|\lambda|)^{\frac{1}{2}})((2k+n)|\lambda|)^{\frac{1}{2}}\rfloor $, and consider
		$$\Pi_{j=1}^N R_k^{n-1}(\lambda, f_j) \leq \Pi_{j=1}^N c_n (\rho_j \sqrt{(2k+n)|\lambda|})^{-n+1/2} $$
		where we have used the estimates proved in Lemma \ref{est-four-lem}. As $ \{\rho_j\} $ is decreasing
		\begin{equation}\label{eq5} \Pi_{j=1}^N c_n (\rho_j \sqrt{(2k+n)|\lambda|})^{-n+1/2} \leq c_n^N    \big(\rho_N\sqrt{(2k+n)|\lambda| }\big)^{-(n-1/2)N}.\end{equation}
		By the choice of $ \rho_j $, it follows that 
		$$\rho_N^2(2k+n)|\lambda| \geq  c_n^4 e^4 \frac{\Theta(N)^2}{N^2} (2k+n)|\lambda| . $$
		As $ \Theta $ is decreasing and $ N \leq \sqrt{(2k+n)|\lambda|)} $, we have $ \Theta(N) \geq \Theta( \sqrt{(2k+n)|\lambda|}) $ and so 
		$$  \Theta(N)^2 (2k+n)|\lambda|  \geq \Theta\big( \sqrt{(2k+n)|\lambda|} \big)^2 (2k+n)|\lambda|  \geq N^2 $$
		which proves that $\rho_N^2(2k+n)|\lambda| \geq  c_n^4 e^4 .$ Using this in (\ref{eq5}) we obtain
		$$  \Pi_{j=1}^N c_n \big(\rho_j \sqrt{(2k+n)|\lambda|}\big)^{-n+1/2} \leq  (c_n^2 e^2)^{-(n-1)N} e^{-N}.    $$
		Finally, as $ N+1 \geq \Theta(((2k+n)|\lambda|)^{\frac{1}{2}})((2k+n)|\lambda|)^{\frac{1}{2}} $, we obtain the estimate (\ref{eq6}).
	\end{proof}
	
	\subsection{Ingham's theorem} We can now complete the proofs of  Theorems \ref{ingh-hei}, and \ref{ingh-hei-variant}. Since half of the theorem has been already proved, as already mentioned in Section 1, we only need to prove the Theorem \ref{ingh-hei-variant}.\\
	
	\textbf{\textit{Proof of Theorem \ref{ingh-hei-variant}}:}
	 Fix $\lambda\neq 0.$ By the hypothesis, $f^{\lambda}$  vanishes on an open set $U_{\lambda}$ in $\mathbb{C}^n.$		First we assume that $\Theta (\lambda)\geq c\, |\lambda|^{-\frac{1}{2}},\,\, |\lambda| \geq 1.$
		In view of Plancherel formula \eqref{wplan} for the Weyl transform, we have $$(2\pi)^n\|L_{\lambda}^mf^{\lambda}\|^2_2=|\lambda|^n\|W_{\lambda}(L_{\lambda}^mf^{\lambda})\|_{HS}^2=|\lambda|^n\|\widehat{f}(\lambda)H(\lambda)^m\|_{HS}^2.$$ Using the formula for Hilbert-Schmidt norm of an operator we have 
		$$(2\pi)^n\|L_{\lambda}^mf^{\lambda}\|^2_2=|\lambda|^n\sum_{\alpha}((2|\alpha|+n)|\lambda|)^{2m}\|\hat{f}(\lambda)\Phi^{\lambda}_{\alpha}\|^2_2.$$ Now, the given condition on the Fourier transform leads to the estimate
		\begin{align}
			\label{est1}
			(2\pi)^n\|L_{\lambda}^mf^{\lambda}\|^2_2\leq& C |\lambda|^n\sum_{\alpha}((2|\alpha|+n)|\lambda|)^{2m}e^{-2\Theta (((2|\alpha|+n)|\lambda|)^{\frac{1}{2}})((2|\alpha|+n)|\lambda|)^{\frac{1}{2}}}\nonumber\\
			\leq & C|\lambda|\sum_{k=0}^{\infty}((2k+n)|\lambda|)^{2m+n-1}e^{-2\Theta(((2k+n)|\lambda|)^{\frac{1}{2}})((2k+n)|\lambda|)^{\frac{1}{2}}}.
		\end{align}
	We  write the last sum as $I_1+I_2$, where 
	\begin{equation*}
		I_1:=\sum_{k\geq 0, (2k+n)|\lambda|\leq m^8}((2k+n)|\lambda|)^{2m+n-1}e^{-2\Theta(((2k+n)|\lambda|)^{\frac{1}{2}})((2k+n)|\lambda|)^{\frac{1}{2}}}, ~\text{and} 
	\end{equation*} 
\begin{equation*}
	I_2:=\sum_{k\geq 0, (2k+n)|\lambda|>m^8}((2k+n)|\lambda|)^{2m+n-1}e^{-2\Theta(((2k+n)|\lambda|)^{\frac{1}{2}})((2k+n)|\lambda|)^{\frac{1}{2}}}.
\end{equation*}
Now, we estimate each sum separately. Notice that when $(2k+n)|\lambda|\leq m^8,$ we have $ \Theta(((2k+n)|\lambda|)^{\frac{1}{2}})\geq \Theta(m^4)$, as $\Theta$ is decreasing. This shows that 
$$I_1\leq \sum_{k\geq 0, (2k+n)|\lambda|\leq m^8}((2k+n)|\lambda|)^{2m+n-1}e^{-2\Theta(m^4)((2k+n)|\lambda|)^{\frac{1}{2}}} $$
which can be dominated by 
\begin{align*}
	&\sum_{k\geq 0, (2k+n)|\lambda|\leq m^8}((2k+n)|\lambda|)^n\int_{(2k+n)|\lambda|}^{(\sqrt{(2k+n)|\lambda|}+1)^2}x^{2m-1}e^{-2\Theta(m^4)(\sqrt{x}-1)}dx \\ &\leq e^{2\Theta(m^4)}m^{8n}\int_{0}^{\infty}x^{2m-1}e^{-2\Theta(m^4)\sqrt{x}}dx.
\end{align*}
		The change of variable $y=2\Theta(m^4)\sqrt{x}$ transform the last expression into  
		$$e^{2\Theta(m^4)}\frac{m^{8n}}{(2\Theta(m^4))^{4m}}\int_{0}^{\infty}y^{4m-1}e^{-y}dy=e^{2\Theta(m^4)}\frac{m^{8n}}{(2\Theta(m^4))^{4m}}\Gamma(4m).$$
		This along with the fact that $\Theta(m^4)\leq \Theta(1)$ shows that 
		\begin{equation*}
			I_1\leq C \frac{m^{8n}}{(2\Theta(m^4))^{4m}}\Gamma(4m).
		\end{equation*}
	Using  Stirling's formula (see Ahlfors \cite{A}) $ \Gamma(x) = \sqrt{2\pi}\, x^{x-1/2} \,e^{-x} e^{\theta(x)/12x} ,  0< \theta(x) <1 $, which is  valid for $ x > 0 ,$ for large $m$, we observe that
	\begin{equation}
		\label{i1}
		I_1\leq C\left(\frac{2m}{\Theta(m^4)}\right)^{4m}.
	\end{equation}
Now, to estimate $I_2$, we make use of the initial assumption that $\Theta(t)\geq c\, t^{-1/2}$ for $t\geq 1.$ Following the same procedure as above, we observe that $I_2$ is dominated by 
\begin{align*}
	&\sum_{k\geq 0, (2k+n)|\lambda|> m^8}\int_{(2k+n)|\lambda|}^{((2k+n)|\lambda|)^2}x^{2m+n-1}e^{-2c\sqrt{x}}dx\\ &\leq e^{-c\,m^4}\sum_{k\geq 0, (2k+n)|\lambda|> m^8}\int_{(2k+n)|\lambda|}^{((2k+n)|\lambda|)^2}x^{2m+n-1}e^{-c\sqrt{x}}dx \\ & =e^{-c\,m^4}\int_{0}^{\infty} x^{2m+n-1}e^{-2c\sqrt{x}}dx.
\end{align*}
		Again the change of variables $y=c\sqrt{x}$ transforms the above integral into 
		$$2c^{-(4m+2n-2)}\int_{0}^{\infty}y^{4m+2n-1}e^{-y}dy=2c^{-(4m+2n-2)}\Gamma(4m+2n).$$ Hence, we obtain $$I_2\leq 2c^{-(4m+2n-2)}\Gamma(4m+2n)e^{-cm^4}. $$
		Now, for large $m$, using the fact that $\Gamma(4m+2n)\leq \Gamma(5m)$, and Stirling's formula we have 
		$$I_2\leq C(c^{-4}(5m)^5e^{-cm^3})^m.$$
		But the right hand side of above goes to zero as $m\rightarrow \infty.$ Hence, in view of \eqref{i1}, for large $m$, we conclude that 
		\begin{equation}
			I_1+I_2\leq   C\left(\frac{2m}{\Theta(m^4)}\right)^{4m}
		\end{equation} 
	which from \eqref{est1} yields for large $m$ 
	\begin{align*}
		(2\pi)^n\|L_{\lambda}^mf^{\lambda}\|^2_2\leq C|\lambda|\left(\frac{2m}{\Theta(m^4)}\right)^{4m}.
	\end{align*}
		The hypothesis on $ \Theta, $ namely $\int_{1}^{\infty}\frac{\Theta (t)}{t}dt=\infty,$  implies that $\int_{1}^{\infty}\frac{\Theta (y^4)}{y}dy=\infty.$ Hence, by integral test we get $\sum_{m=1}^{\infty}\frac{\Theta(m^4)}{m}=\infty.$ Therefore, it follows that  $$\sum_{m=1}^{\infty}\|L_{\lambda}^mf^{\lambda}\|^{-\frac{1}{2m}}_2=\infty.$$ Since $f^{\lambda}$ vanishes on an open set, by the Theorem \ref{ch-L-strong} ( analogue of Chernoff's theorem for $L_{\lambda}$) we conclude that $f^{\lambda}=0$ which is true for all $\lambda\neq 0.$ Hence $f=0.$ 
		
		Now, we consider the general case. The function $\Psi(y)= (1+|y|)^{-1/2} $ satisfies  $\int_{1}^{\infty}\frac{\Psi(y)}{y}dy<\infty.$ By Theorem \ref{construct}  we can construct a compactly supported  radial function $ F \in L^2(\mathbb{H}^n)$  such that $$\hat{F}(\lambda)^\ast \hat{F}(\lambda)\leq e^{-2\Psi (\sqrt{H(\lambda)})\sqrt{H(\lambda)}},\ \lambda\neq 0.$$ We can further arrange that $\text{supp}(F)\subset B_{\mathbb{C}^n}(0,\delta)\times (-a, a)$ for some $\delta, a>0.$  We now consider the function $h=f\ast F.$  Notice that $$h^{\lambda}(z)=(f\ast F)^{\lambda}(z)=\int_{\C^n}f^{\lambda}(z-w)F^{\lambda}(w)e^{\frac{i\lambda}{2}Im(z.\bar{w})}dw.$$ As $ f^{\lambda} $ is assumed  to vanish on $U_{\lambda}$, the function $h^{\lambda}$ vanishes on a smaller open set $U_{\lambda,\delta}\subset U_{\lambda}.$ We now claim that 
		$$  \widehat{h}(\lambda)^\ast \widehat{h}(\lambda) \leq e^{- 2 \Phi(\sqrt{H(\lambda)}) \sqrt{H(\lambda)}} $$
		where $ \Phi(y) = \Theta(y)+\Psi(y).$ As $ \widehat{h}(\lambda) = \widehat{f}(\lambda)\widehat{F}(\lambda) $, for any $ \varphi \in L^2(\R^n) $ we have
		$$ \langle \widehat{h}(\lambda)^\ast \widehat{h}(\lambda) \varphi, \varphi \rangle = \langle \widehat{f}(\lambda)^\ast \widehat{f}(\lambda) \widehat{F}(\lambda) \varphi, \widehat{F}(\lambda)\varphi \rangle .$$
		The hypothesis on $ f $ gives us the estimate 
		$$ \langle \widehat{f}(\lambda)^\ast \widehat{f}(\lambda) \widehat{F}(\lambda) \varphi, \widehat{F}(\lambda)\varphi \rangle  \leq C \langle e^{- 2 \Theta(\sqrt{H(\lambda)}) \sqrt{H(\lambda)}} \widehat{F}(\lambda) \varphi, \widehat{F}(\lambda)\varphi \rangle .$$
		As $ F $ is radial, $ \widehat{F}(\lambda) $ commutes with any function of $ H(\lambda)$ and hence the right hand side can be estimated using the decay of $ \widehat{F}(\lambda)$:
		$$ \langle  \widehat{F}(\lambda)^\ast \widehat{F}(\lambda) e^{- \Theta(\sqrt{H(\lambda)}) \sqrt{H(\lambda)}}\varphi,  e^{- \Theta(\sqrt{H(\lambda)}) \sqrt{H(\lambda)}}\varphi \rangle 
		\leq C  \langle   e^{-2 (\Theta+\Psi)(\sqrt{H(\lambda)}) \sqrt{H(\lambda)}}\varphi,  \varphi \rangle .$$
		This proves our claim on $ \widehat{h}(\lambda) $ with $ \Phi= \Theta + \Psi.$ As $ \Phi(y) \geq |y|^{-1/2} $, by the already proved
		part of the theorem we conclude that $h=0.$ In order to conclude that $f=0$ we proceed as follows.

		Given $ F $ as above, let us consider $  \delta_rF(z,t) = F(rz,r^2t).$ It has been shown elsewhere (see e.g. \cite{LT}) that
		$$ \widehat{\delta_rF}(\lambda) = r^{-(2n+2)}  d_r \circ \widehat{F}(r^{-2} \lambda) \circ d_r^{-1} $$
		where $ d_r $ is the standard dilation on $ \R^n$ given by $ d_r\varphi(x) = \varphi(rx).$ The property of the function $ F ,$ namely $\hat{F}(\lambda)^\ast \hat{F}(\lambda)\leq e^{-2\Psi (\sqrt{H(\lambda)})\sqrt{H(\lambda)}}$ gives us 
		$$ \widehat{\delta_rF}(\lambda)^\ast \widehat{\delta_rF}(\lambda) \leq C r^{-2(2n+2)}    d_r \circ e^{-2\Psi (\sqrt{H(\lambda/r^2)})\sqrt{H(\lambda/r^2)}} \circ d_r^{-1} .$$
		Testing against $ \Phi_\alpha^\lambda $ we can simplify the right hand side which gives us
		$$ \widehat{\delta_rF}(\lambda)^\ast \widehat{\delta_rF}(\lambda) \leq C r^{-2(2n+2)}      e^{-2\Psi_r (\sqrt{H(\lambda)})\sqrt{H(\lambda)}}, $$ 
		where $ \Psi_r(y) = \frac{1}{r}\Psi(y/r).$ If we let $ F_\varepsilon(x) = \varepsilon^{-(2n+2)} \delta_{\varepsilon^{-1}}F(x) $, then it follows that $F_\varepsilon $ is an approximate identity. Moreover, $ F_\varepsilon $ is compactly supported and satisfies the same hypothesis as $ F $ with $\Psi(y)$ replaced by $ \varepsilon \Psi(\varepsilon y)$ which has the same integrability and decay conditions. Hence, working with $F_{\varepsilon}$ we can conclude that $f\ast F_{\varepsilon} = 0$ for any $\varepsilon > 0.$ Letting $\varepsilon\rightarrow0$ and
		noting that $f \ast F_{\varepsilon}$ converges to $f$ in $L^1(\He^n)$, we conclude that $f = 0.$ This completes the
		proof. \qed
\begin{rem} It would be interesting to see whether the conclusion of the Theorem \ref{ingh-hei-variant} still holds true under the assumption that the function vanishes on a non-empty open subset of $\He^n.$ A moment's thought staring at the above proof reveals that this can be achieved if we use an analogue of the Theorem \ref{ch-euc} for the sublaplacian instead of special Hermite operators. But it turns out that proving an analogue of Theorem \ref{ch-euc} is a very interesting and difficult open problem. We hope to revisit this in the near future.
\end{rem}
	
	\section*{Acknowledgments}  The work of the first named author is supported by INSPIRE Faculty Awards from the Department of Science and Technology. The second author is supported by Int.Ph.D. scholarship from Indian Institute of Science. The third named author is supported by NBHM Post-Doctoral fellowship from the Department of Atomic Energy (DAE), Government of India. The work of the last named author  is supported by  J. C. Bose Fellowship from the Department of Science and Technology, Government of India.\\

\end{document}